\documentclass[12pt]{article}

\usepackage{amsmath}
\usepackage{amssymb}
\usepackage{amsthm}
\usepackage{mathabx}
\usepackage[utf8]{inputenc}
\usepackage[T1]{fontenc}
\usepackage[english]{babel}
\usepackage[top=4cm, bottom=4cm, left=3cm, right=3cm]{geometry}
\usepackage{tikz}
\usepackage{multirow}
\usepackage{hyperref}
\usepackage{cleveref} 
\usepackage{paralist}
\graphicspath{ {figures/} }

\tikzstyle boundaries=[color=gray, very thin]
\tikzstyle entrance path=[color=black, very thick]
\tikzstyle pipes=[very thick, rounded corners=5pt]

\newtheorem{theorem}{Theorem}[section]
\newtheorem{algo}[theorem]{Algorithm}
\newtheorem{prop}[theorem]{Proposition}
\newtheorem{lemma}[theorem]{Lemma}

\newtheorem{defi}[theorem]{Definition}

\theoremstyle{remark}
\newtheorem{ex}{Example}[section]
\newtheorem{rem}[theorem]{Remark}

\newcommand{\rc}{\mathbf{R}}
\DeclareMathOperator{\lin}{lin}
\DeclareMathOperator{\SC}{SC}
\DeclareMathOperator{\inv}{inv}

\DeclareMathOperator{\ninv}{ninv}
\DeclareMathOperator{\ins}{ins}
\DeclareMathOperator{\id}{id}

\newcommand{\R}{\mathbb{R}} 
\newcommand{\Z}{\mathbb{Z}} 
\newcommand{\fS}{\mathfrak{S}} 

\newcommand{\less}{\vartriangleleft} 
\newcommand{\more}{\vartriangleright} 

\newcommand{\pos}{^{-1}}

\newcommand{\contact}{^\#} 
\newcommand{\econtact}[1]{#1^\natural}
\newcommand{\pstart}[1]{\mathcal{S}_{#1}}
\newcommand{\pend}[1]{\mathcal{E}_{#1}}
\newcommand{\paths}[1]{\mathcal{#1}}
\newcommand{\ftop}{^\uparrow} 
\newcommand{\fbot}{^\downarrow} 

\newcommand{\cross}{\begin{tikzpicture}[scale=0.3] \draw[pipes] (0.5,0) -- (0.5,1); \draw[pipes] (0,0.5) -- (1,0.5); \end{tikzpicture}}
\newcommand{\elbow}{\begin{tikzpicture}[scale=0.3] \draw[pipes] (0.5,0) -- (0.5,0.5) -- (1,0.5); \draw[pipes] (0,0.5) -- (0.5,0.5) -- (0.5,1); \end{tikzpicture}}
\newcommand{\nwelbow}{\begin{tikzpicture}[scale=0.3] \draw[pipes] (0,0.5) -- (0.5,0.5) -- (0.5,1); \end{tikzpicture}}
\newcommand{\seelbow}{\begin{tikzpicture}[scale=0.3] \draw[pipes] (0.5,0) -- (0.5,0.5) -- (1,0.5); \end{tikzpicture}}
\newcommand{\zone}[1]{\mathcal{Z}_{#1}}
\newcommand{\vertline}[1]{\Delta^{\updownarrow}(#1)}
\newcommand{\horline}[1]{\Delta^{\leftrightarrow}(#1)}

\title{Lattice properties of acyclic alternating pipe dreams}
\author{N. Cartier}

\begin{document}

\maketitle

\begin{abstract}
This paper proves some conjectures raised by N. Bergeron, N. Cartier, C. Ceballos, and V. Pilaud about acyclic facets of subword complexes of type A Coxeter groups. We use a representation of subword complex facets extending the pipe dreams defined N. Bergeron and S. Billey and prove that the strongly acyclic facets realize a lattice quotient of a weak order interval. We also give a sufficient condition so that all acyclic facets are strongly acyclic.
\end{abstract}

\tableofcontents

\section*{Introduction}

Pipe dreams were introduced by N. Bergeron and S. Billey in \cite{BergeronBilley} as objects indexing the monomials of Schubert polynomials. There were later extended to a wider framework in the context of Gröbner geometry by A. Knutson and E. Miller and generalized to subword complexes on Coxeter groups \cite{KnutsonMiller-subwordComplex,KnutsonMiller-GroebnerGeometry}. The structure of subword complexes as simplical complexes and the flip order defined on them have been extensively studied (see for example \cite{PilaudStump-ELlabeling}).

As observed by A. Woo in \cite{Woo}, the sets of reversing pipe dreams (triangular pipe dreams with well-chosen exit permutations) belong to the Catalan families; moreover, the flip order on them is isomorphic to the Tamari lattice on binary tree. This lattice is a quotient of the weak order on permutations, with the insertion into binary search trees being a lattice morphism from the weak order to the Tamari lattice \cite{HivertNovelliThibon-algebraBinarySearchTrees}. This allowed V. Pilaud and F. Santos to find a realization of the associahedron as the Brick polytope of a subword complex \cite{PilaudSantos-brickPolytope}. With a slight variations on the exit permutation, C. Stump found a similar link between other pipe dreams sets and multitriangulations \cite{Stump}, thus prompting N. Bergeron, the author, C. Ceballos and V. Pilaud to study the relation between the weak order and all sets of triangular pipe dreams \cite{CartierPilaud}. They found that acyclic triangular pipe dreams define quotients of the weak order, and conjectured that this result was true for a large family of subword complexes on all Coxeter group.

The goal of the present paper is to prove this conjecture in type A. This is done by representing the facets of the relevant subword complexes by pipe dreams on a well-chosen class of shapes. Those shapes include the classical triangular shape as well as the shapes used by V. Pilaud \cite{Pilaud-brickAlgebraArxiv} to realize cambrian lattices with pipe dreams.

In \cref{sec:definitions}, we extend the definition of pipe dreams to draw them on what we call alternating shapes. This generalizes the usual pipe dreams on a triangular shape, but also the Cambrian twists defined in \cite{Pilaud-brickAlgebraArxiv} (in the preprint version only), shown to be isomorphics to the cambrian lattices introduced by N. Reading \cite{Reading-CambrianLattices}.  This also gives a graphical representation of the subword complexes described in conjectures 5.20 and 5.22 to 5.25 
in \cite{CartierPilaud}. We give a few basic properties of those pipe dreams and of their contact graphs.

In \cref{sec:lattice}, we choose any alternating shape~$F$ and permutation~$\omega$ and study the linear extensions of the contact graphs of the pipe dreams drawn on~$F$ and with exit permutation~$\omega$. We use a result of \cite{CartierPilaud} to prove that they define a partition of a set containing the weak order interval~$[\id, \omega]$; we then prove that the equivalence relation defined by this partition is a lattice congruence of this interval. Finally, we define and study the map from~$[\id, \omega]$ to those pipe dreams such that the fiber above a pipe dream contains its linear extensions: we describe the image of the weak order and we give two algorithms computing the image of any~$\pi \in [\id, \omega]$.

Last, in \cref{sec:cshapes}, we give a sufficient condition on the alternating shape~$F$ so that this map that we defined is surjective on all acyclic pipe dreams on~$F$, independently of the chosen exit permutation~$\omega$. We prove this by reasoning on alternating shapes and pipe dreams, but this result is also a consequence of the much more general Theorem 3.1 of \cite{JahnStump}.

\section{Definitions}

\label{sec:definitions}

\subsection{Weak order on permutations}

Let~$n \geqslant 2$ be an integer and~$\omega \in \fS_n$ a permutation. For any~$1 \leqslant a < b \leqslant n$, the pair~$(a,b)$ is a (right) \emph{inversion} of~$\omega$ if~$\omega\pos(b) < \omega\pos(a)$ and a (right) \emph{noninversion} otherwise. We denote by~$\inv(\omega)$ the set of inversions of~$\omega$ and by~$\ninv(\omega)$ the set of noninversions of~$\omega$. We note that~$\ninv(\omega) \sqcup \inv(\omega) = \{ (a,b) \mid 1\leqslant a < b \leqslant n \}$. The \emph{length} of~$\omega$, denoted by~$\ell(\omega)$, is the number of inversions of~$\omega$.

The (right) \emph{weak order} on permutations is the order given by the inclusion of inversion sets: for any~$\omega_1, \omega_2 \in \fS_n$, we have that~$\omega_1 \leqslant \omega_2$ if and only if~$\inv(\omega_1) \subseteq \inv(\omega_2)$, or equivalently~$\ninv(\omega_2) \subseteq \ninv(\omega_1)$. The covers of this order are all in the form~$UabV \lessdot UbaV$ with~$U$ and~$V$ two sequences of integers and~$1 \leqslant a < b \leqslant n$, or~$\omega \lessdot \omega \tau_i$ when~$\ell(\omega) < \ell(\omega\tau_i)$ for some~$1 \leqslant i < n$ (with~$\tau_i$ the simple transposition~$(i,i+1)$).

Any comparison between two permutations in this paper will be in the right weak order.

\subsection{Alternating shapes and pipe dreams}

Consider the cartesian grid on~$\R^2$, whose cells are the~$1 \times 1$ squares with all their corners in~$\Z^2$. We index those cells by the coordinates of their lower left corner. We describe paths on this grid, i.e.\ sequences of points such that two consecutive points are at distance~1, by giving their starting point and the direction of each step:~$N$,~$S$,~$E$ and~$W$ represent respectively a step north, south, east and west. For a path~$\paths{P}$, we denote by~$|\paths{P}|$ its number of steps or length, and for~$d \in \{ N,S,E,W \}$ a direction we denote by~$|\paths{P}|_d$ the number of steps in that direction in~$\paths{P}$.

\begin{defi}
An alternating shape $F$ is a connected collection of cells of the cartesian grid whose boundary can be divided in four parts as follows for some integer~$n > 0$:

\begin{compactitem}
\item a \emph{starting path}~$\pstart{F}$ from~$(0,0)$ to~$(|\pstart{F}|_E, -|\pstart{F}|_S)$ with~$n$ steps~$S$ or~$E$;
\item a \emph{NW stair path} from~$(0,0)$ to~$(t_F,t_F)$ with steps~$(NE)^{t_F}$ for some~$t_F \geqslant 0$;
\item an \emph{ending path}~$\pend{F}$ from~$(t_F,t_F)$ to~$(t_F+|\pend{F}|_E, t_F - |\pend{F}|_S)$ with~$n$ steps~$S$ or~$E$;
\item a \emph{SE stair path} from~$(|\pstart{F}|_E, -|\pstart{F}|_S)$ to~$(t_F+|\pend{F}|_E, t_F - |\pend{F}|_S)$ with steps~$(EN)^{b_F}$ for some~$b_F \geqslant 0$.
\end{compactitem}
For $F$ to be connected, $\pend{F}$ must stay strictly north and east of $\pstart{F}$ all along.
\end{defi}

We note that since $|\pstart{F}|_S + |\pstart{F}|_E = |\pend{F}|_S + |\pend{F}|_E = n$, the end points of $\pstart{F}$ at coordinates $(|\pstart{F}|_E, -|\pstart{F}|_S)$ and of $\pend{F}$ at coordinates $(t + |\pend{F}|_E, t - |\pend{F}|_S)$ are on the same diagonal of equation $x-y = n$ and so can be joined by a stair path alternating steps $E$ and $N$.

An example of such a shape for~$n=5$ is given on the left of \cref{ex:shape}, with its starting and ending paths are bolded and in red. The second and third shapes are not alternating shapes: the northwest part of the second one is not a stair path, and the northeast part of the third one has an~$N$ step.

\begin{figure}
\begin{center}
\includegraphics[scale=1]{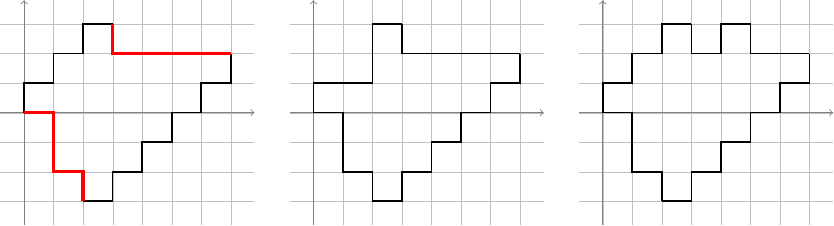}
\end{center}
\caption{An alternating shape and two counter examples.}
\label{ex:shape}
\end{figure}

\begin{defi}
A \emph{pipe dream}~$P$ on an alternating shape~$F$ is a filling of the cells in~$F$ with crosses~\cross{} and contacts~\elbow{} such that each pipe entering on the starting path of~$F$ exit on the ending path. We label the pipes from~1 to~$n$ following the staring path from north-west to south-east, and we say that the order of their exit points on the ending path (still from north-west to south-east) is the \emph{exit permutation} of~$P$. For example, the pipe dreams in \cref{ex:reduced} have exit permutation~$51324$. A permutation~$\omega \in \fS_n$ is \emph{sortable} on~$F$ if there exists at least one pipe dream on~$F$ with exit permutation~$\omega$.
\end{defi}

We will only consider \emph{reduced} pipe dreams, where any two pipes cross at most once, and for any~$\omega \in \fS_n$ sortable on~$F$ we denote by~$\Pi_F(\omega)$ the set of reduced pipe dreams on~$F$ with exit permutation~$\omega$. The crosses in a reduced pipe dreams are in bijection with the inversions of its exit permutation. We note that since we can reduce any non-reduced pipe dream by replacing pairs of crosses between two pipes by pairs of contacts without changing its exit permutation, as illustrated in \cref{ex:reduced}, a permutation~$\omega$ is sortable on~$F$ if and only if~$\Pi_F(\omega)$ is nonempty. 

\begin{figure}
\begin{center}
\includegraphics[scale=1]{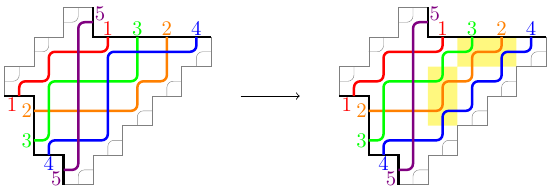}
\end{center}
\caption{A non-reduced pipe dream and its reduced counterpart.}
\label{ex:reduced}
\end{figure}

\begin{rem}\label{rem:n-shapes_alternating}
As was said in the introduction, this definition of pipe dreams generalizes the original definition given in \cite{KnutsonMiller-GroebnerGeometry}. The triangular shape used there is the alternating shape obtained by the starting path having only~$S$ steps, the ending path having only~$E$ steps, the NW stair path having length~0 and the SE stair path length~$2n$.

More generally, reduced pipe dreams on alternating shapes represent the facets of subword complexes in the Coxeter group~$A_{n-1}$ such that the base word is \emph{alternating}, i.e.\ such that between two occurrences of the same letter~$s$, all the letters that do not commute with~$s$ appear. A base word~$Q_F$ represented by an alternating shape~$F$ can be found by reading the shape from south-west to north-east and adding the simple transposition~$\tau_{x-y}$ to~$Q_F$ when encountering the cell~$(x,y)$ of~$F$. Conversely, one can prove that any alternating word~$Q$ on the simple transpositions of~$A_{n-1}$ is represented by an alternating shape.
\end{rem}

Let us now consider~$P \in \Pi_F(\omega)$. A contact~$c$ of~$P$ is \emph{flippable} if the two pipes passing through~$c$ have a crossing~$x$. In that case, the \emph{flip} on~$c$ exchanges the contact~$c$ and the crossing~$x$, and gives a new pipe dream~$P'$ taht is also in~$\Pi_F(\omega)$. The flip is \emph{increasing} if~$c$ is south-west of~$x$ and \emph{decreasing} otherwise. An increasing flip is illustrated in \ref{ex:flip}.

\begin{figure}
\begin{center}
\includegraphics[scale=1]{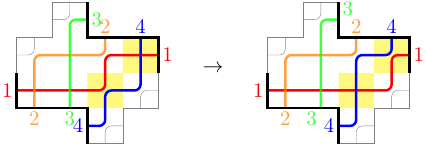}
\caption{An increasing flip between two reduced pipe dreams.}
\label{ex:flip}
\end{center}
\end{figure}

The \emph{increasing flip graph} on~$\Pi_F(\omega)$ is the directed graph whose edges are the increasing flips. It is clearly acyclic and it has a unique source and a unique sink, respectively called the \emph{antigreedy} and \emph{greedy} pipe dreams. The \emph{increasing flip poset} on~$\Pi_F(\omega)$ is the reflexive and transitive closure of the increasing flip graph.

\begin{rem}\label{rem:is_SC}
The increasing flip order on~$\Pi_F(\omega)$ is equivalent to the flip order on the subword complex~$\SC(Q_F,\omega)$ with~$Q_F$ defined in \cref{rem:n-shapes_alternating}.
\end{rem}


Let us now introduce some notations to describe the trajectory of pipes in a pipe dream.

\begin{defi}\label{def:pipes_area}
Let~$F$ be an alternating shape and~$P$ a pipe dream on~$F$. For any pipe~$p$ of~$P$:
\begin{compactitem}
\item $(x_p^s, y_p^s)$ the \emph{starting coordinates} of~$p$ point to the SW corner of its first cell;
\item $(x_p^e, y_p^e)$ the \emph{ending coordinates} of~$p$ point to the NE corner of its last cell.
\end{compactitem}

We say that the rectangle that has~$(x_p^s,y_p^s)$ as its southwest corner and~$(x_p^e,y_p^e)$ as its northeast corner is the \emph{zone} of~$p$, denoted by~$\zone{p}$.
\end{defi}

\begin{rem} \label{rem:zones}
The area of~$p$ contains the entire trajectory of~$p$: any cell crossed by~$p$ must be in it. Since it contains at least the first cell of~$p$, we know that~$x_p^s < x_p^e$ and~$y_p^s < y_p^e$. We also note that since the pipes start along the starting path of~$F$ in increasing order from northwest to southeast, for any~$1 \leqslant a < b \leqslant n$ we know that~$0 \leqslant x_a^s \leqslant x_b^s \leqslant |\pstart{F}|_E$ and~$0 \geqslant y_a^s \geqslant y_b^s \geqslant -|\pstart{F}|_S$. Similarly, since the pipes end along the ending path of~$F$ from northwest to southeast in the order of the exit permutation~$\omega$, we have~$t_F \leqslant x_{\omega(a)}^e \leqslant x_{\omega(b)}^e \leqslant t_F + |\pend{F}|_E$ and~$t_F - |\pend{F}|_S \leqslant y_{\omega(b)}^e \leqslant y_{\omega(a)}^e \leqslant t_F$.
\end{rem}

\begin{lemma}\label{lem:start_end_coord}
For~$F$ an~alternating shape and~$P$ a pipe dream on~$F$ with exit permutation~$\omega$,
\begin{compactitem}
\item $x_p^s - y_p^s = p$ if~$p$ starts horizontally and~$p-1$ otherwise;
\item $x_p^e - y_p^e = \omega\pos(p)$ if~$p$ ends vertically and~$\omega\pos(p)-1$ otherwise.
\end{compactitem}
\end{lemma}

\begin{proof}
We denote the coordinates of the integer points of~$\pstart{F}$ from northwest to southeast by~$(x_0,y_0)$, $(x_1,y_1), \ldots, (x_n,y_n)$ and know that~$x_0=y_0=0$; since this path is made of south and east steps, for each~$0 \leqslant i < n$ we have either~$x_{i+1} = x_i$ and~$y_{i+1} = y_i - 1$, or~$x_{i+1} = x_i +1$ and~$y_{i+1} = y_i$. Since~$x_0-y_0 = 0$, this proves that for all~$i$ we have~$x_i - y_i = i$. Then for~$1 \leqslant p \leqslant n$, pipe~$p$ starts on the~$p$-th step of~$\pstart{F}$ which is between points~$(x_{i-1},y_{i-1})$ and~$(x_i,y_i)$. The cell in which~$p$ starts is then indexed by its southwest corner, so~$(x_{i-1},y_{i-1})$ if the~$p$-th step goes east and~$(x_i,y_i)$ if it goes south. The first case corresponds to~$p$ starting vertically and the second case to~$p$ starting vertically.

The second part of the lemma is similar: if~$(x'_i,y'_i)$ is the~$i$-th integer point of~$\pend{F}$ (indexed from~$0$ to~$n$) then since~$x'_0 = y'_0 = t_F$ and~$\pend{F}$ is made of south and east steps, we have~$x'_i - y'_i = i$ for each~$i$. Then since pipe~$p$ ends on the~$\omega\pos(p)$-th step of~$\pend{F}$ and its ending coordinates point to the northeast corner of its last cell, depending on the direction of that~$\omega\pos(p)$-th step (and so the ending direction of pipe~$p$) the ending coordinates of~$p$ are either~$(x'_{\omega\pos(p)-1},y'_{\omega\pos(p)-1})$ or~$(x'_{\omega\pos(p)},y'_{\omega\pos(p)})$, thus concluding the proof.
\end{proof}

\subsection{Contact graphs}

We will now define the contact graph of a pipe dream, which is the specification to our context of the contact graph of a pseudoline arrangement first introduced in \cite{PilaudSantos-brickPolytope}.

\begin{defi}
For~$P$ a pipe dream on an alternating shape~$F$ with exit permutation~$\omega \in \fS_n$, the \emph{contact graph} of~$P$ is the directed graph~$P\contact$ that has a vertex for each pipe of~$P$ and contains the arc~$(a,b)$ if and only if a contact
\begin{tikzpicture}[scale=0.5, baseline=4pt]
\draw[pipes,red] (0,0.5) node[left] {a} -- (0.5,0.5) -- (0.5,1);
\draw[pipes,blue] (0.5, 0) -- (0.5,0.5) -- (1,0.5) node[right] {b};
\end{tikzpicture}
appears somewhere in~$P$. The \emph{extended contact graph} of~$P$, denoted by~$\econtact{P}$, is obtained from~$P\contact$ by adding the missing arcs~$(a,b)$ such that~$a < b$ and~$\omega\pos(a) < \omega\pos(b)$ (i.e.~$(a,b)$ is a noninversion of~$\omega$). A pipe dream~$P$ is \emph{acyclic} if~$P\contact$ is acyclic and \emph{strongly acyclic} if~$\econtact{P}$ is acyclic.
\end{defi}

The pipe dream given in \cref{ex:bad_fiber} is acyclic but not strongly acyclic, as seen in its associated contact graph and extended contact graph.

\begin{figure}
\begin{center}
\includegraphics[scale=1]{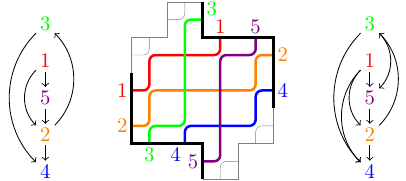}
\caption{The contact graph and extended contact graph of a pipe dream.}
\label{ex:bad_fiber}
\end{center}
\end{figure}

For~$F$ an alternating shape and~$\omega \in \fS_n$ sortable on~$F$, we denote by~$\Sigma_F(\omega)$ the set of reduced, strongly-acyclic pipe dreams on~$F$ with exit permutation~$\omega$. For any pipe dream~$P$ in that set, we denote by~$\less_P$ the transitive closure of~$\econtact{P}$ on~$[n]$ and by~$\lin(P)$ the linear extensions of~$\econtact{P}$.

Let us now give our first two lemmas on contact graphs.

\begin{lemma}\label{lem:elbows_rectangle}
Let $F$ be an alternating shape, $P$ a reduced pipe dream on $F$ and~$p,q$ two pipes of~$P$. Suppose that pipe~$p$ has an elbow in a cell~$(x_p,y_p)$ weakly northwest of an elbow of~$q$ in a cell~$(x_q,y_q)$. If any of the following three conditions is met, then there is a directed path from~$p$ to~$q$ in~$P\contact$:
\begin{compactenum}
\item the cells~$(x_p,y_q)$ and~$(x_q,y_p)$ (the other corners of the rectangle between the two elbows) are in~$F$;
\item there exists~$(x_1,y_1)$ and~$(x_2,y_2)$ cells of~$F$ such that the two elbows are both north-east of~$(x_1,y_1)$ and south-west of~$(x_2,y_2)$;
\item $(x_p,y_p) \in \zone{q}$.
\end{compactenum}
\end{lemma}

Some cases of this lemma are illustrated in \cref{fig:elbows_rectangle}. The first pair of elbows has the rectangle between them partially outside of the shape, so the lemma cannot be applied, and there is no path from~$5$ to~$3$ in the contact graph. The second pair of elbows has the rectangle between them completely in the shape, so we can apply case one of the lemma, and we can check that there is a path~$1 \to 3 \to 2 \to 4$ in the contact graph. Finally, the third pair of elbows is between the two gray cells, both inside the shape, and so we can apply case~2 of the lemma; once again, there is a path~$3 \to 2 \to 4$ in the contact graph.

\begin{figure}
\begin{center}
\includegraphics[scale=.9]{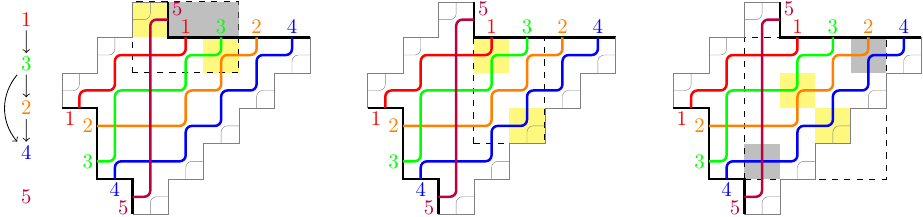}
\caption{Illustrations of the cases of \cref{lem:elbows_rectangle}.}
\label{fig:elbows_rectangle}
\end{center}
\end{figure}

\begin{proof}
We will start by proving the first case, and the second and third cases are natural consequences of it. Suppose thus that the situation is as described in the lemma, and that~$(x_p,y_q)$ and~$(x_q,y_p)$ are in~$F$. We proceed by induction on the grid distance between~$(x_p,y_p)$ and~$(x_q,y_q)$. If that distance is~$0$, then the two elbows are in the same cell, so there is an edge~$p \to q$ in~$P\contact$ and so~$p \less_P q$. Otherwise, let~$p'$ be the pipe with a southeast elbow \seelbow{} in~$(x_p,y_p)$, and~$q'$ be the pipe with a northwest elbow \nwelbow{} in~$(x_q,y_q)$ (both must exist, since from the respective positions of~$(x_p,y_p)$ and~$(x_q,y_q)$ the former cannot be on the SE stair path of~$F$ nor the latter be on the NW stair path). We know that either~$p = p'$ or~$p \to p'$ is an edge of~$P\contact$, and either~$q = q'$ or~$q' \to q$ is an adge of~$P\contact$. Let us now follow~$p'$ along row~$y_p$ and column~$x_p$. Either~$p'$ crosses straight through both cells~$(x_p,y_q)$ and~$(x_q,y_p)$ or it has an elbow northwest of~$(x_q,y_q)$ strictly closer to that cell than~$(x_p,y_p)$. In the second case, by induction hypothesis we obtain a directed path from~$p'$ to~$q$ in~$P\contact$. In the first case, we can similarly follow~$q'$ along row~$y_q$ and column~$x_q$; either its crosses straight through both cells~$(x_p,y_q)$ and~$(x_q,y_p)$ or as for~$p'$ we can apply the induction hypothesis to obtain a directed path from~$p$ to~$q'$ in~$P\contact$. Since~$P$ is reduced, we know that~$p'$ and~$q'$ cannot cross twice in cells~$(x_p,y_q)$ and~$(x_q,y_p)$, so we must be able to apply the induction hypothesis. Thus in all cases, we obtain a directed path from~$p$ to~$q$ in~$P\contact$.

Suppose now that we are in case~2 with~$(x_1,y_1)$ and~$(x_2,y_2)$ as described. From the positions of the cells, we know that~$x_1 \leqslant x_p \leqslant x_q \leqslant x_2$ and~$y_1 \leqslant y_q \leqslant y_p \leqslant y_2$. Since the upper boundary of~$F$ is made up of the NW stair path and of the ending path, it goes north then south when sweeping~$F$ from west to east; in particular, since column~$x_q$ is between column~$x_p$ and column~$x_2$, that boundary cannot be lower in column~$x_q$ than in both column~$x_p$ and column~$x_2$. Therefore, since the cells~$(x_p,y_p)$ and~$(x_2,y_2)$ are both in~$F$, the boundary in column~$x_q$ is at least on row~$\min(y_p,y_2) = y_p$. The cell~$(x_q,y_p)$ is therefore south of the upper boundary of~$F$ and north of its cell~$(x_q,y_q)$, and so it is in~$F$. Similarly, since the lower boundary of~$F$ goes south (on the starting path) then north (on the SE stair path) when sweeping~$F$ from west to east, it cannot be higher in column~$x_p$ than in both column~$x_1$ and~$x_q$. It must thus be below row~$\max(y_1,y_q) = y_q$, and so cell~$(x_p,y_q)$ is below~$(x_p,y_p)$ and above the lower boundary, so it is in~$F$. Since both~$(x_p,y_q)$ and~$(x_q,y_p)$ are in~$F$, we can then apply case 1.

Finally, case~3 is a direct consequence of case~2 with the first and last cell of pipe~$q$ used for~$(x_1,y_1)$ and~$(x_2,y_2)$.
\end{proof}

From this lemma we deduce the following statement, which will be at the heart of our later proofs.

\begin{lemma}\label{lem:3_inv}
Let $F$ be an alternating shape and $P$ a reduced pipe dream on $F$ with exit permutation $\omega \in \fS_n$, and $1 \leqslant p < q < r \leqslant n$ such that $\omega\pos(r) < \omega\pos(q) < \omega\pos(p)$. Then if there is an contact of $P$ involving the pipes $p$ and $r$, one of the two following statements is true:
\begin{compactitem}
\item $q \less_P p$ and $q \less_P r$;
\item $p \less_P q$ and $r \less_P q$.
\end{compactitem}
\end{lemma}

\begin{proof}
We know that~$(p,q),(q,r),(p,r) \in \inv(\omega)$ so the three pairs must all cross at some point in~$P$. Denote by~$x$ the cell where~$p$ and~$q$ cross, and~$x'$ the cell where~$q$ and~$r$ cross. Since pipe~$q$ goes through both of those cells, one must be southwest of the other; for now, we suppose that~$x$ is southwest of~$x'$. Moreover, since pipe~$q$ starts and ends between pipes~$p$ and~$r$, any contact between~$p$ and~$r$ must be between~$x$ and~$x'$; suppose that such a contact exist and denote it by~$c$ ($x$ is SW of~$c$ which is SW of~$x'$).

Since pipe~$q$ goes through~$x$ vertically (as~$q > p$) and through~$x'$ horizontally (as~$q < r$), it must have at least one elbow between those two points; we denote by~$e$ the first elbow of~$q$ after~$x$, which must therefore be in the same column as~$x$ and north of it, and also be southwest of~$x'$. Similarly, since pipe~$p$ has a contact~$c$ after~$x$, we can consider its first elbow~$e_p$ after~$x$; it must be in the same row as~$x$ and east of it, and also be weakly southwest of~$c$. Then~$e$ is strictly northwest of~$e_p$ and both elbows are northeast of~$x$ and southeast of~$x'$, so by case~2 of \cref{lem:elbows_rectangle} we obtain that there is a path from~$q$ to~$p$ in~$P\contact$.

Similarly, the last elbow~$e'$ of~$q$ before~$x'$ is directly west of~$x'$ and northeast of~$x$, and the last elbow~$e_r$ of~$r$ before~$x'$ is directly south of~$x'$ and northeast of~$c$, so also northeast of~$x$. Since~$e'$ is northwest of~$e_r$ and both are between~$x$ and~$x'$, we apply case~2 of \cref{lem:elbows_rectangle} and obtain that there is a path from~$q$ to~$r$ in~$P\contact$.

See \cref{fig:proof_3inv} for an illustration of the placement of the various cells that we used. The reasoning is similar for~$x'$ southwest of~$x$, and in that case we obtain paths from~$p$ and~$r$ to~$q$ in~$P\contact$.
\end{proof}

\begin{figure}
\begin{center}
\includegraphics[scale=1]{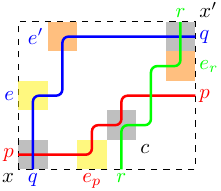}
\caption{Illustration of the proof of lemma \cref{lem:3_inv}.}
\label{fig:proof_3inv}
\end{center}
\end{figure}

\section{A lattice on strongly acyclic pipe dreams}

\label{sec:lattice}

Let~$P$ be a strongly acyclic pipe dream on the alternating shape~$F$ with exit permutation~$\omega$. A \emph{linear extension} of~$P$ is a permutation in~$\fS_n$ that is a topological ordering of~$\econtact{P}$. We denote by~$\lin(P)$ the set of those linear extensions.

We note that since~$\econtact{P}$ has an arc for each noninversion of~$\omega$, any element of~$\lin(P)$ is below~$\omega$ in the right weak order on~$\fS_n$. More specifically, the permutations in~$\lin(P)$ are exactly the topological orderings of~$P\contact$ that are below~$\omega$ in the right weak order.

\subsection{Linear extensions of pipe dreams}

In this section we show that linear extensions of pipe dreams on the same alternating shape and with the same exit permutation behave very nicely on~$[\id, \omega]$.

\begin{theorem}\label{thm:lin_union}
Let~$F$ be an alternating shape and~$\omega\in \fS_n$ be sortable on~$F$. Then the set~$\{ \lin(P) \mid P \in \Sigma_F(\omega) \}$ is a partition of the weak order interval~$[\id, \omega]$.
\end{theorem}

To prove this theorem, we will use the following result, given in section 6 
of \cite{CartierPilaud}.

\begin{theorem}\label{thm:Coxeter_lin_union}
Let~$\SC(Q,w)$ be a non-empty subword complex.
The following hold:
\begin{compactenum}
\item  $\bigcup_{I\in\SC(Q,w)} \lin(I)$ is a lower set of the weak order.
\label{thmA_item1}
\item $[e,w] \subseteq \bigcup_{I\in\SC(Q,w)} \lin(I)$ 
\label{thmA_item2}
\item If $I_1\neq I_2$ then $\lin(I_1)\cap \lin(I_2)=\varnothing$.
\label{thmA_item3}
\item The set $\lin(I)$ is closed by intervals.
\label{thmA_item4}
\end{compactenum}
\end{theorem}

\begin{proof}
As noted in \cref{rem:is_SC}, we can see the set of reduced pipe dreams~$\Pi_F(\omega)$ as the facets of a subword complex~$\SC(Q_F, \omega)$ on the Coxeter group~$A_{n-1}$, with~$Q_F$ determined by~$F$. The root configuration of the facet associated to a pipe dream~$P$ is given by its contact graph~$P\contact$, and thus acyclic pipe dreams are associated to acyclic facets.

Since~$\econtact{P}$ contains~$P\contact$, we know that any linear extension of~$\econtact{P}$ is also a linear extension of~$P\contact$ and is therefore in~$\lin(\rc(P))$. In particular, combined with \cref{thmA_item3} of \cref{thm:Coxeter_lin_union}, for any two pipe dreams~$P,P' \in \Sigma_F(\omega)$, either~$P = P'$ or~$\lin(P) \cap \lin(P') \subset \lin(\rc(P)) \cap \lin(\rc(P')) = \emptyset$, so the sets~$(\lin(P))_{P\in \Sigma_F(\omega)})$ are disjoints.

Moreover, since the arcs in~$\econtact{P}$ that are not in~$P\contact$ all correspond to noninversions of~$\omega$, any~$\pi \leqslant \omega$ that is a linear extension of~$P\contact$ (i.e.\ of~$\rc(P)$) is also a linear extension of~$\econtact{P}$, and if such a~$\pi$ exists in~$\lin(P)$ then~$P$ is strongly acyclic. Therefore, since from \cref{thmA_item2} of \cref{thm:Coxeter_lin_union}~$[\id, \omega] \subset \bigcup_{P \in \Pi_F(\omega)} \lin(\rc(P))$, we know that~$[\id, \omega] \subset \bigcup_{P \in \Sigma_F(\omega)} \lin(P)$.

Finally, since for any~$P \in \Sigma_F(\omega)$ the graph~$\econtact{P}$ contains an arc for each noninversion of~$\omega$, any linear extension of~$\econtact{P}$ is also in~$[\id, \omega]$, and~$\bigcup_{P \in \Sigma_F(\omega)} \lin(P) \subset [\id, \omega]$. This conclude the proof.
\end{proof}

This theorem shows that for any permutation~$\pi \in [\id, \omega]$ there exists exactly one pipe dream~$P \in \Sigma_F(\omega)$ such that~$\pi \in \lin(P)$. The following result shows that if we know this~$P$ for~$\pi$, then we can easily find~$P'$ such that~$\pi' \in \lin(P')$ for any~$\pi' \in [\id, \omega]$ covering~$\pi$.

\begin{lemma}\label{lem:cover_flip}
If~$\pi' := UpqV \lessdot \pi := UqpV$ is a cover of the weak order and~$\pi$ is a linear extension of~$P\contact$ for some pipe dream~$P \in \Sigma_F(\omega)$, then:
\begin{compactitem}
\item if~$P\contact$ has no arc~$q \to p$ then~$\pi'$ is a linear extenions of the same graph;
\item otherwise~$\pi'$ is a linear extension of~$P'{}\contact$, with~$P'$ the pipe dream obtained by flipping the northeastmost contact between pipes~$q$ and~$p$ in~$P$.
\end{compactitem}
\end{lemma}

\begin{proof}
If there is no arc~$q \to p$ in~$P\contact$, then~$\pi'$ is obviously also a linear extension of that same graph; otherwise there is at least one contact between~$q$ and~$p$ in~$P$. Since~$q$ starts southeast of~$p$ and is northwest of~$p$ on that contact, pipes~$p$ and~$q$ must cross at some point in~$P$, so any contact between them is flippable. If we flip the one that is furthest northeast, we reverse all the contacts between~$p$ and~$q$ in~$P$ and we only exchange~$p$ and~$q$ in some other contacts with other pipes. This guarantees that the obtained pipe dream~$P'$ has~$\pi'$ as a linear extension of its contact graph.
\end{proof}

\subsection{A lattice congruence}

Let us denote by~$\equiv_{F,\omega}$ the equivalence relation on~$[\id, \omega]$ that has~$\{ \lin(P) \mid P \in \Sigma_F(\omega) \}$ as its equivalence classes. By \cref{thm:lin_union}, this definition is correct. This section will prove the following theorem.

\begin{theorem}\label{thm:lattice_congruence}
For any alternating shape~$F$ and any permutation~$\omega \in \fS_n$ sortable on~$F$, the relation~$\equiv_{F,\omega}$ is a lattice congruence of the weak order interval~$[\id, \omega]$.
\end{theorem}

As a reminder, an equivalence relation~$\equiv$ on a lattice~$(L, \leqslant, \wedge, \vee)$ is a lattice congruence if and only if it respects meets and joins, i.e.\ for any~$x, x', y, y' \in L$, if~$x \equiv x'$ and~$y \equiv y'$ then~$x \wedge y \equiv x' \wedge y'$ and~$x \vee y \equiv x' \vee y'$. To prove \cref{thm:lattice_congruence}, we will use the following classical characterization of lattice congruences.

\begin{theorem}[\cite{Reading-latticeTheory}]\label{thm:congruence_charac}
An equivalence relation~$\equiv$ on a lattice~$L$ is a congruence if and only if
\begin{compactenum}
\item every equivalence class of~$\equiv$ is an interval;
\item the projections~$p\ftop: L \mapsto L$ and~$p\fbot: L \mapsto L$, which map an element of~$L$ respectively to the maximal and minimal elements of its equivalence class, are order preserving.
\end{compactenum}
\end{theorem}

We will start with point 1 of the characterization of lattice congruences.

\begin{prop}\label{prop:fibers_intervals}
For any alternating shape~$F$ and any~$\omega \in \fS_n$ sortable on~$F$, the equivalence classes of~$\equiv_{F, \omega}$ are intervals.
\end{prop}

To prove it, we will use the following classical characterization of weak order intervals:

\begin{prop}[{\cite[Thm. 6.8]{BjornerWachs}}]\label{prop:WOIP}
The set of linear extensions of a poset~$\less$ on~$[n]$ is an interval~$I$ of the weak order if for every~$1 \leqslant i < j < k \leqslant n$,
\[
i \less k \implies i \less j \text{ or } j \less k
\qquad\text{and}\qquad
i \more k \implies i \more j \text{ or } j \more k.
\]
Moreover, the inversions of~$\min(I)$ are the pairs~$i,j \in [n]$ with $i < j$ and $i \more j$, and the non-inversions of~$\max(I)$ are the pairs~$i,j \in [n]$ with $i < j$ and $i \less j$.
\end{prop}


\begin{proof}[Proof of \cref{prop:fibers_intervals}]
Consider~$P$ a pipe dream of~$\Sigma_F(\omega)$ and suppose that for some pipes~$p < q < r$ we have~$p \less_P r$. We will prove by induction on the length of the shortest path from~$p$ to~$r$ in~$\econtact{P}$ that~$p \less_P q$ or~$q \less_P r$. Suppose first that this path is of length one, i.e.\ that~$p \rightarrow r$ is an arc of~$P\contact$. Then either~$\omega\pos(p) < \omega\pos(q)$, or~$\omega\pos(q) < \omega\pos(r)$, or~$\omega\pos(r) < \omega\pos(q) < \omega\pos(p)$. In the first and second cases, since~$(p,q)$ or~$(q,r)$ is in~$\ninv(\omega)$, from the definition of~$\econtact{P}$ we know that~$p \less_P q$ or~$q \less_P r$. In the third case, since there exists a contact from pipe~$p$ to pipe~$r$, \cref{lem:3_inv} tells us that either~$p \less_P q$ or~$q \less_P r$.

Suppose now that our statement is true for any distance strictly lower than~$d > 1$ and that the distance from~$p$ to~$r$ is~$d$, and denote by~$p'$ the first pipe on this path after~$p$. Then~$p \less_P p' \less_P r$ and the distance from~$p'$ to~$r$ is~$d-1$. Then:
\begin{compactitem}
\item if~$p' = q$ then~$p \less_P q$;
\item if~$p' < q$ then~$p' < q < r$ and~$p' \less_P r$ so by our induction hypothesis, either~$q \less_P r$ or~$q \more_P p' \more_P p$;
\item if~$p' > q$, then~$p < q \less_P p'$ and~$p \less_P p'$ with a path of length~1 from~$p$ to~$p'$ in~$P\contact$, so by induction hypothesis either~$p \less_P q$ or~$q \less_P p' \less_P r$.
\end{compactitem}
In all cases, the result still holds for distance~$d$. By induction, it is thus always true. 

A symmetrical reasoning proves that if~$r \less_P p$, then~$r \less_P q$ or~$q \less_P p$. We can now apply \cref{prop:WOIP} to obtain that the set~$\lin(P)$ is an interval. Since all congruence classes of~$\equiv_{F,\omega}$ are defined as the linear extension set of some pipe dream in~$\Sigma_F(\omega)$, this concludes the proof.
\end{proof}

Note that \cref{prop:WOIP} tells us that the equivalence classes of~$\equiv_{F,\omega}$ have a minimal and a maximal element, but also a characterization of these: the inversions of~$\min(\lin(P))$ are exactly the pairs~$p < q$ such that~$q \less_P p$, and the noninversions of~$\max(\lin(P))$ are the pairs~$p < q$ such that~$p \less_P q$.

\begin{prop}\label{prop:minmax}
Let~$C,C'$ be equivalence classes of~$\equiv_\omega$ and consider~$\pi \in C$ and~$\pi' \in C'$. Then~$\pi \leqslant \pi'$ implies that~$\max(C) \leqslant \max(C')$ and~$\min(C) \leqslant \min(C')$.
\end{prop}

\begin{proof}
We prove the statement for the maximums, the proof for the minimums is symmetrical.
Observe first that we can assume that~$\pi$ is covered by~$\pi'$ in weak order, so that we write~${\pi' = \pi \tau_i}$ for some simple transposition~$\tau_i := (i \; i+1)$.
The proof now works by induction on the weak order distance between~$\pi$ and~$\max(C)$.
If~$\pi = \max(C)$, the result is immediate as~${\max(C) = \pi < \pi' \le \max(C')}$.
Otherwise, must be covered~$\pi$ by a permutation~$\sigma$ in the class~$C$, and we write~$\sigma = \pi \tau_j$ for some simple transposition~$\tau_j$.
Let~$P,P' \in \Sigma(\omega)$ be such that~$C = \lin(P)$ and~$C' = \lin(P')$.
We now distinguish five cases, according to the relative positions of~$p$ and~$q$:
\begin{enumerate}[(1)]
\item If~$i > j+1$, then~$\pi = UpqVrsW$, $\pi' = UpqVsrW$ and~$\tau = UqpVrsW$ for some~$p < q$~and~$r < s$. Define~$\sigma' := \pi \tau_i \tau_j = \pi \tau_j \tau_i = UqpVsrW$. By \cref{lem:cover_flip}, there is no arc~$p \to q$ in~$\econtact{P}$ (since~$\pi$ and~$\sigma$ both belong to~$C$), and~$\econtact{P}$ and~$\contact{P'{}}$ can only differ by arcs incident~to~$r$~or~$s$. Hence, there is no arc~$p \to q$ in~$\econtact{P'{}}$. We thus obtain again by \cref{lem:cover_flip} that~$\sigma' \in \lin(P') = C'$.
\item If~$i = j+1$, then~$\pi = UpqrV$, $\pi' = UprqV$ and~$\sigma = UqprV$ for some~$p < q < r$. Define~$\sigma' := \pi \tau_i \tau_j \tau_i = \pi \tau_j \tau_i \tau_j = UrqpV$. Since~$\pi \in \lin(P)$, we have~$p \not\more_P q$ and~$q \not\more_P r$, so that there is no arc~$p \to r$ in~$\econtact{P}$ by \cref{lem:3_inv}. By \cref{lem:cover_flip}, there is no arc~$p \to q$ in~$\econtact{P}$, and~$\econtact{P}$ and~$\econtact{P'{}}$ can only differ by arcs incident to~$q$ or~$r$. We thus obtain that there is no arc~$p \to q$ nor~$p \to r$ in~$\econtact{P'{}}$. Consequently, again by \cref{lem:cover_flip}, both~$\pi'\tau_j$ and~$\sigma' = \pi'\tau_j \tau_i$ belong to~$\lin(P') = C'$.
\item If~$i = j$, then~$\pi' = \sigma$ is in~$C$, so that~$C = C'$ and there is nothing to prove.
\item If~$i = j-1$, we proceed similarly as in Situation~(2).
\item If~$i < j-1$, we proceed similarly as in Situation~(1).
\end{enumerate}
In all cases, we found~$\sigma' > \sigma$ with~$\sigma' \in C'$.
Since~$\sigma < \sigma'$ with~$\sigma \in C$ and~$\sigma' \in C'$, and since~$\sigma$ is closer to~$\max(C)$ than~$\pi$, we obtain that ${\max(C) < \max(C')}$ by induction hypothesis.
\end{proof}

We have thus proved that the top and bottom projections defined in \cref{thm:lattice_congruence} are order preserving.

\begin{proof}[Proof of \ref{thm:lattice_congruence}]
The theorem follows immediately by combining \cref{thm:congruence_charac}, \cref{prop:fibers_intervals} and \cref{prop:minmax}.
\end{proof}

\subsection{The insertion map and the acyclic order}

Let us go back to \cref{thm:lin_union}: since the sets of linear extensions of pipe dreams of~$\Sigma_F(\omega)$ for some alternating shape~$F$ and some permutation~$\omega$ sortable on~$F$ is a partition of~$[\id, \omega]$, it means that for any~$\pi \in [\id, \omega]$, there exists exactly one~$P \in \Sigma_F(\omega)$ such that~$\pi \in \lin(P)$. We can thus define a map from~$[\id, \omega]$ to~$\Sigma_F(\omega)$.

\begin{defi}
Let~$F$ be an alternating shape and~$\omega$ be sortable on~$F$. We denote by~$\ins_{F,\omega}:[\id, \omega] \mapsto \Sigma_F(\omega)$ the map such that for any~$\pi \in [\id, \omega]$, we have~$\pi \in \lin(\ins_{F,\omega}(\pi))$.
\end{defi}

We note that by \cref{lem:cover_flip}, for any~$\pi_1, \pi_2 \in [\id, \omega]$, if~$\pi_1 \leqslant \pi_2$ then~$\ins_{F,\omega}(\pi_1) \leqslant \ins_{F,\omega}(\pi_2)$ in the increasing flip order on~$\Pi_F(\omega)$. However, the image of the weak order by~$\ins_{F,\omega}$ on~$\Sigma_F(\omega)$ is not the ascending flip order, as illustrated in \cref{ex:bad_flip}: there is an ascending flip from~$P_1$ to~$P_2$, but the only linear extension~$132$ of~$P_1$ is not comparable in the weak order to the only linear extension~$231$ of~$P_2$. We thus have to define a new order relation on~$\Sigma_F(\omega)$.

\begin{figure}
\begin{center}
\includegraphics[scale=1]{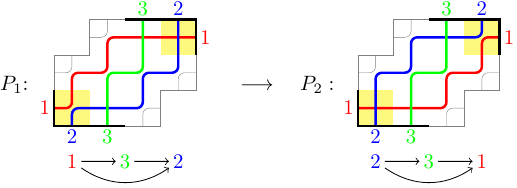}
\end{center}
\caption{An increasing flip between two strongly acyclic pipe dreams that is not the image of a weak order relation.}
\label{ex:bad_flip}
\end{figure}

\begin{defi}
For any~$P,P' \in \Sigma_F(\omega)$, we say that~$P \leqslant P'$ in the \emph{acyclic order} if there exists~$\pi \in \lin(P)$ and~$\pi' \in \lin(P')$ such that~$\pi \leqslant \pi'$ in the weak order.
\end{defi}

This is obviously by definition the image of the weak order by~$\ins_{F,\omega}$. The following proposition gives alternative characterizations of this order that does not use linear extensions.

\begin{prop}\label{prop:def_acyclic_order}
For any pipe dreams~$P,P' \in \Sigma_F(\omega)$, the following statements are equivalent:
\begin{compactenum}[(1)]
\item there exists~$\pi \in \lin(P)$ and~$\pi' \in \lin(P')$ such that~$\pi \leqslant \pi'$;
\item the minimal (resp. maximal) linear extensions~$\pi$ of~$P$ and~$\pi'$ of~$P'$ satisfy~$\pi \leqslant \pi'$;
\item there are no pipes~$p < q$ such that~$p \more_P q$ and~$p \less_{P'} q$;
\item for all pipes~$p < q$, if~$p \more_P q$ then~$p \more_{P'} q$;
\item for all pipes~$p < q$, if~$p \less_{P'} q$ then~$p \less_P q$.
\end{compactenum}
\end{prop}

\begin{proof}
Both versions of item~$(2)$ obviously imply item~$(1)$, and from \cref{prop:minmax} the converse is true, so~$(1) \iff (2)$. The equivalence~$(2) \iff (3) \iff (4) \iff (5)$ comes from the description of the inversions of the bottom elements and noninversions of the top elements of~$\lin(P)$ and~$\lin(P')$ given by \cref{prop:WOIP}, and the definition of the weak order in term of inclusion of inversion sets.
\end{proof}

We know very well the structure of the weak order on~$[\id, \omega]$ but very little yet of its image, the acyclic order. In particular, we can characterize the covers of the weak order, and we know that it is a lattice. The following two results will show how those properties translate to the acyclic order.

\begin{prop}\label{prop:acyclic_covers}
For any two distinct pipe dreams~$P_1,P_2 \in \Sigma_F(\omega)$, $P_2$ covers~$P_1$ in the acyclic order if and only if there exists~$\pi_1 \in \lin(P_1)$ and~$\pi_2 \in \lin(P_2)$ such that~$\pi_1 \lessdot \pi_2$ is a cover of the weak order.
\end{prop}

\begin{proof}
Suppose that~$P_2$ covers~$P_1$ and choose two permutations~$\pi_1 < \pi_2$ respectively in~$\lin(\pi_1)$ and~$\lin(P_2)$, then there exist~$\pi_1 = \sigma_0 \lessdot \sigma_1 \lessdot \ldots \lessdot \sigma_k = \pi_2$ a maximal chain from~$\pi_1$ to~$\pi_2$ for some~$k \geqslant 1$. If for any~$0 \leqslant j \leqslant k$ we have~$\ins_{F,\omega}(\sigma_j)= P \notin \{ P_1,P_2 \}$, then by definition of the acyclic order~$P_1 < P < P_2$ and~$P_2$ does not cover~$P_1$; therefore, the whole chain is in~$\lin(P_1) \cup \lin(P_2)$. Since it starts in~$\lin(P_1)$ and ends in~$\lin(P_2)$, there exists~$j$ such that~$\sigma_j \in \lin(P_1)$ and~$\sigma_{j+1} \in \lin(P_2)$, and~$\sigma_{j+1}$ covers~$\sigma_j$.

Suppose now that~$\pi_1 \lessdot \pi_2$ is a cover of the weak order and~$P_1 := \ins_{F,\omega}(\pi_1)$ is distinct from~$P_2 := \ins_{F,\omega}(\pi_2)$. By definition of the acyclic order~$P_1 < P_2$; consider~$P \in \Sigma_F(\omega)$ such that~$P_1 \leqslant P \leqslant P_2$ and denote by~$\pi\fbot,\pi\ftop$ the minimal and maximal elements of~$\lin(P)$. Then since~$P \leqslant P_2$ we know that~$\pi\fbot \leqslant \pi_2$ and we chose~$\pi_1 \leqslant \pi_2$, so~$\pi_1 \leqslant \pi_1 \vee \pi\fbot \leqslant \pi_2$; since~$\pi_2$ covers~$\pi_1$, this means that~$\pi_1 \vee \pi\fbot$ is either~$\pi_1$ or~$\pi_2$. In the first case, this means that~$\pi\fbot \leqslant \pi_1$ and so~$\pi_1 \in [\pi\fbot,\pi\ftop] = \lin(P)$, so from \cref{thm:lin_union} this means that~$P = P_1$. In the second case, since~$P_1 \leqslant P$ we know that~$\pi_1 \leqslant \pi\ftop$ and by definition~$\pi\fbot \leqslant \pi\ftop$ so~$\pi_2 = \pi_1 \vee \pi\fbot \leqslant \pi\ftop$, thus~$\pi_2 \in [\pi\fbot,\pi\ftop] = \lin(P)$ and so in the same way~$P = P_2$. Thus~$P_1 < P < P_2$ is not possible, and so~$P_2$ covers~$P_1$ in the weak order.
\end{proof}

This results shows that the covers of the weak order give exactly the covers of the acyclic order.

\begin{rem}
The flips starting from~$P \in \Sigma_F(\omega)$ that are covers of the acyclic orders are the one between two pipes~$p$ and~$q$ such that~$(p,q)$ is an edge of the contact graph~$P\contact$, but there is no other path from~$p$ to~$q$ in~$P\contact$. This corresponds to the rays of the root cone of the subword complex facet associated to~$P$, and so the flips that are covers are the ones corresponding to edges of the brick polyhedron of the subword complex, as described in \cite{JahnStump}.
\end{rem}

\begin{theorem}
For~$F$ an alternating shape and~$\omega \in \fS_n$ sortable on~$F$, the acyclic order defines a lattice on~$\Sigma_F(\omega)$.
\end{theorem}

\begin{proof}
Consider~$P_1,P_2 \in \Sigma_F(\omega)$ and write~$\lin(P_1) = [\pi_1\fbot,\pi_1\ftop]$ and~$\lin(P_2) = [\pi_2\fbot,\pi_2\ftop]$. We call~$P\fbot = \ins_{F,\omega}(\pi_1\ftop \wedge \pi_2\ftop)$ and~$P\ftop = \ins_{F,\omega}(\pi_1\fbot \vee \pi_2\fbot)$; by definition, it is clear that~$P\fbot \leqslant P_1,P_2$ and~$P\ftop \geqslant P_1,P_2$. Then for~$P \in \Sigma_F(\omega)$, if~$P \leqslant P_1$ and~$P \leqslant P_2$ then for~$\pi = \min(\lin(P))$ we know that~$\pi \leqslant \pi_1\ftop$ and~$\pi \leqslant \pi_2\ftop$, so~$\pi \leqslant \pi_1\ftop \wedge \pi_2\ftop$ and so~$P \leqslant P\fbot$; therefore~$P\fbot$ is the maximum of all pipe dreams below both~$P_1$ and~$P_2$, and so~$P\fbot = P_1 \wedge P_2$. A similar reasoning proves that~$P\ftop = P_1 \vee P_2$, and so any pair of elements of~$\Sigma_F(\omega)$ has both a join and a meet and the acyclic order defines a lattice on~$\Sigma_F(\omega)$.
\end{proof}

\subsection{Two algorithms}

The following section will give two ways of computing~$\ins_{F,\omega}(\pi)$ for~$\omega$ sortable on~$F$ and~$\pi \leqslant \omega$: one by filling each cell of the shape from south-west to north-east, and one by determining the path of each pipe one by one.

The sweeping algorithm creates a pipe dream by sweeping the alternating shape from southwest to northeast; this allows us to know which pipes are entering a celle from the west and the south when considering that cell, as the trajectory of those pipes from their starting point to that cell is already fully determined. We then choose to fill each cell with a contact \elbow{} or a cross \cross{} depending on the relative positions of the entering pipes in~$\pi$ and~$\omega$: with~$p$ at the west and~$q$ at the south, the two pipes can only cross if~$(p,q) \in \inv(\omega)$, in which case they cross as soon as they meet if~$\pi\pos(p) > \pi\pos(q)$ and as late as possible otherwise.

\begin{algo}[Sweeping algorithm]\label{algo:gpd_sweep}
For any~alternating shape~$F$ and permutations~$\pi \leqslant \omega$ with~$\omega$ sortable on~$F$, the pipe dream~$\ins_{F,\omega}(\pi)$ can be constructed by sweeping~$F$ from southwest to northeast. We place a crossing \cross{} when sweeping cell~$c$ if and only if pipe~$p$ arriving horizontally and pipe~$q$ arriving vertically in~$c$ satisfy the following two statements:
\begin{compactitem}
\item $(p,q)$ is an inversion of~$\omega$;
\item $\pi\pos(p) > \pi\pos(q)$, or~$c$ is the last possible crossing point for pipes~$p$ and~$q$.
\end{compactitem}
\end{algo}

We note that determining if a cell~$c$ is the last possible crossing point for two pipes can be done by computing the Demazure product of the remaining cells (excluding~$c$). If this product is greater in the strong Bruhat order than the suffix of~$\omega$ that still needs to be written, then the current cell is not the last possible crossing point for~$p$ and~$q$; otherwise, it is and the cell needs to contain a crossing.

This algorithm is equivalent to the sweeping algorithm on Coxeter subword complexes given and proven in \cite{CartierPilaud}; it is also a variation on algorithm 2.12 given in \cite{JahnStump}. We give an example in \ref{ex:sweeping_algorithm}.

\begin{ex}
The execution of \cref{algo:gpd_sweep} with~$\omega = 23145$ and~$\pi = 21345$ is drawn in \cref{ex:sweeping_algorithm}. On the first step, the incoming pipes of the bottom cell are~$4$ (west) and~$5$ (south); since~$\omega\pos(4) < \omega\pos(5)$, we fill this cell with a contact~\elbow. On the second step, the incoming pipes of the considered are~$1$ and~$2$; since~$\omega\pos(2) < \omega\pos(1)$ and~$\pi\pos(1) > \pi\pos(2)$, we fill this cell with a cross \cross. On the third step, the incoming pipes are~$1$ and~$3$; since~$\pi\pos(1) < \pi\pos(3)$ and those two pipes can still cross later, we fill this cell with a contact \elbow. Finally, on the fifth step, the incoming pipes are once again~$1$ and~$3$ but this time it is the last possible crossing point for them, so we fill this cell with a cross \cross.
\end{ex}

\begin{figure}
\begin{center}
\includegraphics[scale=1]{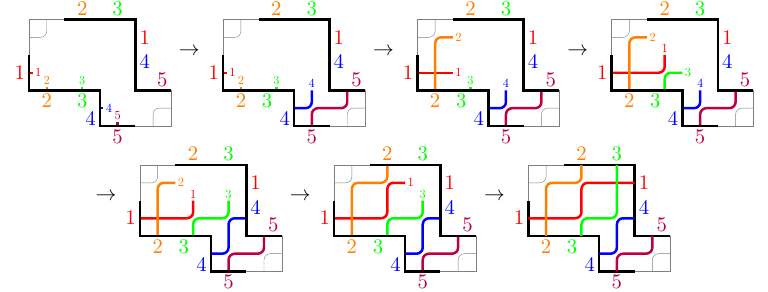}
\end{center}
\caption{Executing the sweeping algorithm for~$\omega=23145$ and~$\pi=21345$.}
\label{ex:sweeping_algorithm}
\end{figure}

The insertion algorithm, on the other hand, creates a pipe dream by determining the full trajectory of each pipe in the order given by~$\pi$, and it does so by keeping track of cells containing only the NW elbow of a contact and completing them with the next pipes.

\begin{algo}[Insertion algorithm]\label{algo:gpd_insertion}\index{pipe dreams!insertion algorithm}
For any~alternating shape~$F$ and permutations~$\pi \leqslant \omega$ with~$\omega$ sortable on~$F$, the pipe dream~$\ins_{F,\omega}(\pi)$ can be obtained by inserting each pipe in the order given by~$\pi$. At step~$i$, we insert a pipe starting on step~$\pi(i)$ of~$\pstart{F}$, ending on step~$\omega\pos(\pi(i))$ of~$\pend{F}$, and whose northeast elbows are precisely completing all the previously created southwest elbows in the zone of pipe~$\pi(i)$.
\end{algo}

This algorithm, illustrated in \cref{ex:insertion_algorithm}, is a generalization to alternating shapes of the insertion algorithm in triangular pipe dreams given in \cite{CartierPilaud}. The proof that it gives a reduced pipe dream contained in~$F$ is similar in spirit, though some care must be taken to show that no matter the starting and ending directions of a pipe, its insertion will never create a cell containing only the SE elbow of a contact. Once this is proven, checking that the result is a pipe dream on~$F$ that has~$\omega$ as an exit permutation and~$\pi$ as a linear extension follows the same proof sketch as the one used in \cite{CartierPilaud}.

\begin{figure}
\begin{center}
\includegraphics[scale=1]{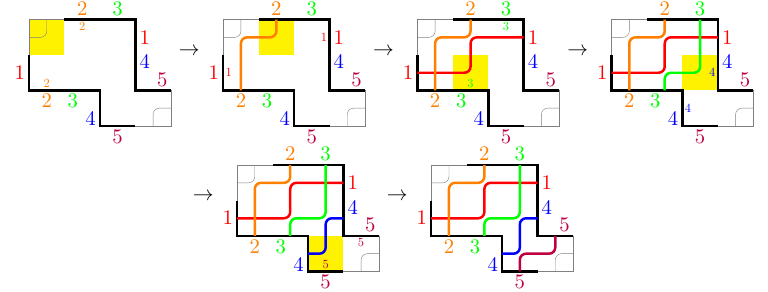}
\end{center}
\caption{The insertion algorithm on a~5-shape with~$\omega = 23145$ and~$\pi=21345$.}
\label{ex:insertion_algorithm}
\end{figure}

\section{Complete shapes}

\label{sec:cshapes}

In the previous section, we showed that for any alternating shape and any permutation sortable on this shape, the map~$\ins_{F,\omega}$ is a lattice morphism from the weak order interval~$[\id, \omega]$ to the acyclic order on the set of strongly acyclic pipe dreams~$\Sigma_F(\omega)$. We also noted that the covers of this poset are given by the edges of the brick polytope on the subword complex~$\SC(F,\omega)$ between strongly acyclic facets. However, as some acyclic pipe dreams are not strongly acyclic, like the one in \cref{ex:bad_fiber}, the Hasse diagram of this poset is not always an orientation of the graph of the brick polytope: some vertices of the polytope may be missing.

In this section, we will provide a sufficient condition on alternating shapes so that any acyclic pipe dreams will also be strongly acyclic. This means that acyclic orders on those shapes will have the graph of their brick polytope as their Hasse diagrams.

\begin{defi}
An alternating shape is \emph{complete} if~$\omega_0 = n \: (n-1) \: \ldots \: 2 \: 1$ the longest permutation of~$\fS_n$ is sortable on it.
\end{defi}

\begin{theorem}\label{thm:cshapes_acyclic}
Let~$F$ be a complete alternating shape, then for any permutation~$\omega \in \fS_n$, all the acyclic pipe dreams in~$\Pi_F(\omega)$ are also strongly acyclic.
\end{theorem}

This theorem is a direct consequence of the following stronger result that we will prove first.

\begin{theorem}\label{thm:cshapes_contact_graph}
Let~$F$ be a complete alternating shape and~$P$ be a reduced pipe dream on~$F$ with~$\omega$ its exit permutation. Then for any~$1 \leqslant p < q \leqslant n$, if~$(p,q) \in \ninv(\omega)$, there is a path from~$p$ to~$q$ in the contact graph~$P\contact$.
\end{theorem}

We will start by giving a few properties of complete shapes in \cref{lem:cshape_coords}, then use them to see how pipes behave in pipe dreams on those shapes in \cref{lem:cshape_ninv,lem:cshape_elbows}. We will then use these last two lemmas to construct a path from~$p$ to~$q$ in the contact graph of a pipe dream on a complete shape, provided that~$(p,q)$ is a noninversion of the exit permutation.

\begin{lemma}\label{lem:cshape_coords}
Let~$F$ be a complete shape and~$P$ be any pipe dream on~$F$ with~$\omega$ its exit permutation, then:
\begin{compactenum}
\item for any pipe~$p$ of~$P$, $x_p^s \leqslant x_n^s < x_{\omega(1)}^e \leqslant x_{p}^e$;
\item for any pipe~$p$ of~$P$, $y_p^s \leqslant y_1^s < y_{\omega(n)}^e \leqslant y_p^e$;
\item $x_n^s \leqslant |\pstart{F}|_E \leqslant x_{\omega(1)}^e$;
\item $y_1^s \leqslant 0 \leqslant y_{\omega(n)}^e$;
\item $|\pstart{F}|_E \leqslant t_F+1$.
\end{compactenum}
\end{lemma}

\begin{proof}
By \ref{rem:zones} we know that~$x_p^s \leqslant x_n^s \leqslant |\pstart{F}|_E$ and~$t_F \leqslant x_{\omega(1)}^e \leqslant x_p^e$, as well as~$y_p^s \leqslant y_1^s \leqslant 0$ and~$y_{\omega(n)}^e \leqslant y_p^e$. We then note that for any pipe dream~$P_0 \in \Pi_F(\omega_0)$, the starting coordinates of pipes are the same in~$P$ and~$P_0$, and since~$\omega_0(1) = n$ and~$\omega_0(n) = 1$, the ending coordinates of~$1$ in~$P_0$ are the same as~$\omega(n)$ in~$P$ and the ones of~$n$ in~$P_0$ are the same as~$\omega_1$ in~$P$. We thus obtains that~$x_n^s < x_\omega(1)^e$ and~$y_1^s < y_{\omega(n)}^e$. The last relations are proved by noting that~$|\pstart{F}|_E - 1 \leqslant x_n^s \leqslant |\pstart{F}|_E$, that~$-1 \leqslant y_1^s \leqslant 0$ and that~$t_F \leqslant y_{\omega(1)}^e \leqslant t_F + 1$.
\end{proof}

In the following proofs, we will denote by~$\vertline{x}$ the vertical line of abscissa~$x$ and~$\horline{y}$ the horizontal line of ordinate~$y$.

\begin{lemma}\label{lem:cshape_ninv}
Let~$F$ be a complete alternating shape and~$P$ be a reduced pipe dream on~$F$ with exit permutation~$\omega$. For~$q$ any pipe of~$P$, if there exists a cell~$c$ of~$F$ strictly northwest of~$\zone{q}$, then there exists a pipe~$p$ with an elbow in a cell strictly southeast of~$c$ and such that~$c \in \zone{p}$, and~$(p,q)$ is a noninversion of~$\omega$.
\end{lemma}

\begin{proof}
We will proceed by counting the number of pipes that have the cell~$c$ in their area, and then we will see that not all of those pipe can pass to the NW of~$c$; therefore, at least one of them must have an elbow SE of that cell and thus be as described by the lemma.

Let us consider a pipe~$q$ of~$F$ and suppose that a cell~$c$ of~$F$ with coordinates~$(x,y)$ is strictly northwest of~$\zone{q}$, i.e.~$x_q^s > x \geqslant y \geqslant y_q^e$. From \cref{lem:cshape_coords} we deduce that any pipe~$p$ satisfies~$x < x_p^e$ and~$y > y_p^s$, so~$(x,y) \in \zone{p}$ if and only if~$x_p^s \leqslant x$ and~$y_p^e > y$. If we call~$X := \{ p \mid x_p^s \leqslant x \}$ and~$Y := \{ p \mid y_p^e > y \}$, this means that~$c \in \zone{p}$ if and only if~$p \in X \cap Y$. Moreover, if~$p \in X \cap Y$, then~$p$ starts strictly west of~$q$ and ends strictly north of~$q$, then~$(p,q)$ is necessarily a noninversion of~$\omega$. 

We know that~$0 \leqslant y_q^e < y \leqslant x < x_p^s$ so the south boundary of~$F$ in columns~$x$ and~$y$ are respectively the~$i$-th and~$j$-th step of~$\pstart{F}$ for some~$1 \leqslant j \leqslant i < p$. Since step~$i$ is the~$(x+1)$-th east step of~$\pstart{F}$ and step~$j$ the~$(y+1)$-th east step, we note that~$i-j \geqslant x - y$. Let us now consider~$P_0 \in \Pi_F(\omega_0)$, and note that according to \cref{lem:start_end_coord} the starting coordinates of pipe~$j$ are~$(y,y-j+1)$; moreover, since pipe~$j$ must cross all the pipes~$1,2, \ldots,j-1$ while going vertically somewhere in~$P_0$, it must go straight vertically through at least~$j-1$ rows, and since it does not end in column~$y$ (because~$y < t_F$) it must have at least one southeast elbow. Therefore, it passes through at least~$j$ different rows of~$F$, and so its ending ordinate is at least~$y+1$. Thus, for any pipe~$p$ of~$P$, if~$\omega\pos(p) \leqslant \omega_0\pos(j) = n+1-j$, then~$y_p^e \geqslant y+1$.

We just proved that~$p \leqslant i \Rightarrow p \in X$ and~$\omega\pos(p) \leqslant n+1-j \Rightarrow p \in Y$. This means that~$|X| \geqslant i$ and~$|Y| \geqslant n+1-j$, and since~$p \notin X \cup Y$, we also know that~$|X \cup Y| \leqslant n-1$. We can then apply the inclusion-exclusion principle to get
\begin{align*}
|X\cap Y| & \geqslant |X| + |Y| - |X \cup Y|\\
& \geqslant i + (n+1-j) - (n-1)\\
& \geqslant i - j + 2\\
|X\cap Y| &\geqslant x-y + 2
\end{align*}

Then each pipe in~$X$ starts west of the vertical line~$\vertline{x+1}$ and ends east of it (since~$x+1 \leqslant x_q^s < x_p^e$ for any~$p$), and each pipe in~$Y$ starts south of the horizontal line~$\horline{y}$ (because~$y \geqslant y_q^e > y_p^s$ for any~$p$) and ends north of it. It can thus cross either first the horizontal line then the vertical one, or the reverse. In the first case, it crosses~$\horline{y}$ west of point~$(x+1,y)$ and inside~$F$, so east of point~$(y,y)$; with one pipe crossing the line in each column, there are at most~$x+1-y$ such pipes. In the second case, the pipe crosses~$\vertline{x+1}$ going horizontally then~$\horline{y}$ going vertically, so there is at least one elbow between those two crossing, and that elbow is southeast of the cell~$c$. Since~$|X\cap Y| > x-y+1$, there is at least one pipe~$p \in X \cap Y$ in the second case, and that pipe satisfies the requirements of the lemma.
\end{proof}

\begin{figure}
\begin{center}
\begin{tikzpicture}[scale=0.3]
\fill[nearly transparent, blue] (9,1) rectangle (12,-2);
\fill[yellow] (6,2) rectangle (7,3);
\foreach \x in {0, ..., 9}
	\draw[boundaries] (\x, \x) -- (\x, \x+1) -- (\x+1, \x+1);
\draw[dashed, blue] (9, -2) node[below right] {$x_q^s$} -- (9, 9);
\draw[dashed, blue] (1,1) -- (12, 1) node[right] {$y_q^e$};
\node[below right] at (9,1) {$\zone{q}$};
\draw[dashed, red] (7, -2) node[below] {$x + 1$} -- (7, 7);
\draw[dashed, red] (2,2) -- (12, 2) node[right] {$y$};
\fill[red] (2,2) circle (5pt) node[above left] {$(y,y)$};
\fill[red] (7,7) circle (5pt) node[above left] {$(x+1,x+1)$};
\end{tikzpicture}
\end{center}
\caption{The part of $F$ north-west of cell $c$ (in yellow).}
\label{ill:cshape_ninv}
\end{figure}
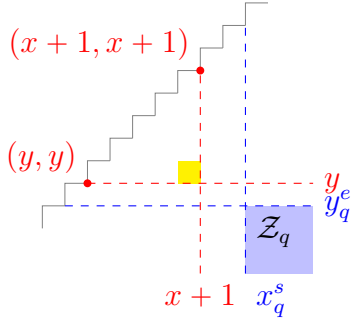

\begin{lemma}\label{lem:cshape_elbows}
Let~$F$ be a complete~alternating shape and~$P \in \Pi_F(\omega)$, and~$(p,q)$ be a noninversion of~$\omega$. Then:
\begin{compactenum}
\item pipe~$p$ has an elbow northwest of an elbow of pipe~$q$;
\item either one such elbow is strictly northwest of~$\zone{q}$, or it is in~$\zone{q}$ and there is a path from~$p$ to~$q$ in~$P\contact$.
\end{compactenum}
\end{lemma}

\begin{proof}
From \cref{lem:cshape_coords} and our choice of~$p$ and~$q$, we know that~$x_p^s \leqslant x_q^s < x_p^e \leqslant x_q^e$ and~$y_q^s \leqslant y_p^s < y_q^e \leqslant y_p^e$. Pipe~$p$ must therefore cross the vertical line~$\vertline{x_q^s}$ and the horizontal line~$\horline{y_q^e}$ (we note that~$p$ cannot start vertically in column~$x_q^s$ or end horizontally in row~$y_q^e$); since it changes direction between those two crossings, it has an elbow in a cell~$(x,y)$ between them. Then:
\begin{itemize}
\item if pipe~$p$ crosses~$\horline{y_q^e}$ first and~$\vertline{x_q^s}$ second, then~$x < x_q^s$ and~$y \geqslant y_q^e$, so this elbow is strictly NW of~$\zone{q}$. Moreover, pipe~$q$ must cross~$\vertline{x_q^e}$ and~$\horline{y_q^s}$ and it changes direction between these crossings, so it has at least one elbow in~$\zone{q}$, and so that elbow is southwest of~$(x,y)$;
\item if pipe~$p$ crosses~$\vertline{x_q^s}$ first and~$\horline{y_q^e}$ second, then~$x_q^s \leqslant x$ and~$y_q^e > y$ so~$(x,y) \in \zone{q}$. Moreover, since~$x < x_p^e \leqslant x_q^e$ and~$y \geqslant y_q^s \geqslant y_p^s$, pipe~$p$ must cross the line~$\vertline{x+1}$ and~$\horline{y}$ at some point, so it has an elbow between thoese two crossing; since~$(p,q) \in \ninv(\omega)$, pipe~$q$ is southeast of pipe~$p$ on its whole trajectory, and so that elbow is southeast of the elbow of~$q$ in cell~$(x,y)$. In that case, we can apply case~$3$ of \cref{lem:elbows_rectangle} to obtain a path from~$p$ to~$q$ in~$P\contact$.
\end{itemize}
\end{proof}

\begin{proof}[Proof of \cref{thm:cshapes_contact_graph}]
We will find a path from~$p$ to~$q$ in~$P\contact$ by using \cref{lem:cshape_ninv} to prove that as long as~$p$ passes northwest of~$\zone{q}$, we can find~$p'$ such that there is a path from~$p$ to~$p'$ and~$(p',q)$ is also a noninversion of~$\omega$. This will create a sequence with its last element in~$\zone{q}$, allowing us to apply the second part of \ref{lem:cshape_elbows}.

Let~$P$ be a reduced pipe dream of~$F$ a complete shape with~$\omega \in \fS_n$ its exit permutation. We choose~$(p,q)$ a noninversion of~$\omega$ and we suppose that for any pipes~$1 \leqslant p' < q' < q \leqslant n$ with~$(p',q') \in \ninv(\omega)$, there is a path from~$p'$ to~$q'$ in~$P\contact$. We will prove that there is also a path from~$p$ to~$q$ in~$P\contact$. To do so, we will build a non-repeating sequence of pipes~$p = p_0, p_1, \ldots, p_N$ and a sequence of cells~$(x_0,y_0), \ldots, (x_{N-1}, y_{N-1})$ of~$F$ such that for all~$0 \leqslant k < N$:
\begin{compactenum}
\item cell~$(x_k, y_k)$ contains an elbow of~$p_k$;
\item $x_k < x_q^s$ and~$y_k \geqslant y_q^s$;
\item the pipe~$p_{k+1}$ has an elbow strictly SE of~$(x_k,y_k)$;
\item there is a path in~$P\contact$ from~$p$ to~$p_{k+1}$;
\item there exists~$l \leqslant k$ such that~$x_{p_{k+1}}^s \leqslant x_l$;
\item there exists~$l \leqslant k$ such that~$y_{p_{k+1}}^s > x_l$;
\end{compactenum}
and pipe~$p_N$ has an elbow in~$\zone{q}$.

We already chose~$p_0 = p$. Suppose now that we have the pipes~$p_0, \ldots, p_k$ and the cells~$(x_0,y_0), \ldots, (x_{k-1}, y_{k-1})$ as described. We already know that either~$k = 0$ and by hypothesis~$(p,q) \in \ninv(\omega)$, or there exists~$l_1,l_2 < k$ such that~$x_{p_k}^s \leqslant x_{l_1} < x_q^s$ and~$y_{p_k}^e > y_{l_2} \geqslant y_q^e$, so pipe~$p_k$ starts strictly west and ends strictly north of pipe~$q$, and so~$(p_k, q) \in \ninv(\omega)$. Therefore, by \cref{lem:cshape_elbows}, pipe~$p_k$ either has an elbow in~$\zone{q}$ and in that case we decide that~$k = N$ and stop the sequence, or it has at least one elbow strictly NW of~$\zone{q}$ in a cell that we denote by~$(x_k,y_k)$ (and thus statements 1 and 2 of our conditions on the sequences are still true).

In that case, let us now denote by~$x = \max_{0 \leqslant l \leqslant k}(x_l)$ and~$y = \min_{0 \leqslant l \leqslant k}(y_k)$, meaning that the cell~$(x,y)$ is the furthest northwest cell that is southeast of every~$(x_l, y_l)$ for~$0 \leqslant l \leqslant k$, as depicted in yellow in \cref{fig:cshapes_proof}. We know that~$x < x_q^s$ and~$y \geqslant y_q^s$, so we can apply \cref{lem:cshape_ninv} to obtain~$p_{k+1}$ with~$(x,y) \in \zone{p_{k+1}}$ and that has an elbow strictly southeast of~$(x,y)$ (and so strictly southeast of every~$(x_l, y_l)$ for~$0 \leqslant l \leqslant k$). This last proposition tells us that~$p_{k+1} \neq p_l$ for all~$0 \leqslant l \leqslant k$, since pipes go east and north and thus no pipe could go through two cells with one strictly SE of the other. Moreover, this choice of~$p_{k+1}$ guarantees that conditions 3, 5 and 6 on our sequences are true. A possible trajectory of pipe~$p_{k+1}$ relative to the elbows in cells~$(x_l, y_l)$ (in black) and to cell~$(x,y)$ (in yellow) is illustrated on the left side of \cref{fig:cshapes_proof}.

\begin{figure} 
\begin{center}
\includegraphics[scale=1]{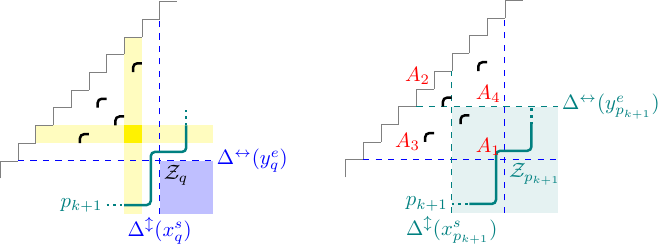}
\end{center}
\caption{Illustration of the placement of the cells $(x_l,y_l)$.}
\label{fig:cshapes_proof}
\end{figure}

Let us now divide the indices~$0 \leqslant l \leqslant k$ in four sets~$A_1, A_2, A_3, A_4$ depending on the placement of the cell~$(x_l, y_l)$ relative to~$\zone{p_{k+1}}$ as depicted on the right side of \cref{fig:cshapes_proof}: we say that~$l$ is in~$A_1$ if~$(x_l, y_l) \in \zone{p_{k+1}}$, in~$A_2$ if it is strictly NE of~$\zone{p_{k+1}}$, in~$A_3$ is it is strictly west but not north of~$\zone{p_{k+1}}$ and in~$A_4$ if it is strictly north but not west of~$\zone{p_{k+1}}$. At least one of those sets cannot be empty; one of the following cases must then be true:
\begin{itemize}
\item If~$A_1 \neq \emptyset$, then for any~$l \in A_1$ we have~$(x_l, y_l) \in \zone{p_{k+1}}$, and pipe~$p_{k+1}$ has an elbow strictly southeast of~$(x,y)$ and so strictly SE of the elbow of~$p_l$ in cell~$(x_l, y_l)$. We can thus apply statement 3 of \cref{lem:elbows_rectangle} to obtain a path from~$p_l$ to~$p_{k+1}$ in~$P\contact$, and since by induction hypothesis there is a path from~$p$ to~$p_l$ in~$P\contact$, this guarantees that condition 4 is true.
\item If~$A_2 \neq \emptyset$, then for any~$l \in A_2$ pipe~$p_l$ starts strictly west and ends strictly north of pipe~$p_{k+1}$ and thus~$(p_l, p_{k+1}) \in \ninv(\omega)$. Since we supposed at the begining that for any noninversion~$(p',q')$ of~$\omega$ such that~$q' < q$ there is a path from~$p'$ to~$q'$ in~$P\contact$, we know that there is a path from~$p_l$ to~$p_{k+1}$ in~$P\contact$, and so since by induction hypothesis there is a path from~$p$ to~$p_l$ condition 4 on our sequences is once again true.
\item Suppose now that~$A_1 = A_2 = \emptyset$. In that case, since there exist~$0 \leqslant l_3,l_4 \leqslant k$ with~$x_{l_4} = x \geqslant x_{p_{k+1}}^s$ and~$y_{l_3} = y < y_{p_{k+1}}^e$, we know that~$l_3 \in A_3$ and~$l_4\in A_4$.
\begin{compactitem}
\item Suppose now that~$0 \in A_3$ and consider~$m = \min(A_4) > 0$. Since~$m$ is in~$A_4$,  we know that~$y_{p_{k+1}}^e \leqslant x_m < x_{p_m}^e$, so~$\omega\pos(p_m) < \omega\pos(p_{k+1})$. Moreover, by our choice of~$m$, any~$l < m$ is in~$A_3$, so~$x_l < x_{p_{k+1}}$; by condition 5 on our sequences, this means that~$x_{p_m}^s < x_{p_{k+1}}^s$ so~$p_m < p_{k+1}$.
\item Similarly, if~$0 \in A_4$, we choose~$m = \min(A_3) > 0$. The same reasoning tells us that because~$m \in A_3$ then~$a_m < a_{k+1}$ and because of condition 6 applied to~$m$ then~$\omega\pos(a_m) < \omega\pos(a_{k+1})$.
\end{compactitem}
In both cases, we have~$m \leqslant k$ such that~$(p_m,p_{k+1})$ is a non-inversion of~$\omega$. As in the case where~$A_2 \neq \emptyset$, this means that there is a path from~$p_m$ to~$p_{k+1}$ in~$P\contact$ and so condition 4 is still true.
\end{itemize}

We have thus proved the existence of such a sequence, and since it is non-repeating and all elements are pipes of~$P$, it is finite. Therefore, this sequence has a last element~$p_N$ such that there is a path from~$p$ to~$p_N$ in~$P\contact$, and by our choice of an endpoint, we know that~$(p_N, q) \in \ninv(\omega)$ and~$p_N$ has an elbow in~$\zone{q}$. By \cref{lem:cshape_elbows}, this tells us that there is a path from~$p_N$ to~$q$ in~$P\contact$, and by transitivity there is a path from~$p$ to~$q$ in~$P\contact$.

This tells us that the set~$S = \{ (p,q) \in \ninv(\omega) \mid \text{there is no path } p \rightarrow^* q \text{ in } P\contact \}$ must be empty, since~$\min_{(p,q) \in S} (q)$ cannot exist. This conclude the proof.
\end{proof}

\begin{proof}[Proof of \cref{thm:cshapes_acyclic}]
A directed graph is acyclic if and only if its transitive closure is acyclic. For~$F$ a complete shape and~$\omega\in \fS_n$ a permutation, let us consider~$P \in \Pi_F(\omega)$; if it is acyclic, then by definition~$P\contact$ is acyclic and so is its transitive closure~$\overline{P\contact}$. We know from \cref{thm:cshapes_contact_graph} that for any~$(p,q) \in \ninv(\omega)$, the edge~$(p,q)$ is in~$\overline{P\contact}$, and so is any edge of~$P\contact$; therefore, the graph~$\econtact{P}$ is a subgraph of~$\overline{P\contact}$ and is also acyclic. Therefore, if~$P$ is acyclic, then it is also strongly acyclic.
\end{proof}

\bibliographystyle{plain}
\bibliography{biblio}

\end{document}